\documentclass[10pt]{amsart}

\usepackage{amssymb,amsmath} 
\usepackage{amsthm}
\usepackage{mathrsfs} 
\usepackage[pagewise]{lineno} 
\newcommand{\lpar}{(}
\newcommand{\rpar}{)}
\newcommand{\mc}{\mathcal}
\newcommand{\BB}{\mathbb B}
\newcommand{\RR}{\mathbb R}
\newcommand{\ZZ}{\mathbb Z}
\newcommand{\ee}{\mathbf{e}}
\newcommand{\zero}{\mathbf{0}}
\newcommand{\one}{\mathbf{1}}

\theoremstyle{plain}
   \newtheorem{thm}{Theorem}[section]
   \newtheorem{lem}[thm]{Lemma}
   \newtheorem{coro}[thm]{Corollary}
   \newtheorem{propo}[thm]{Proposition}
   \newtheorem{claim}[thm]{Claim}
\theoremstyle{definition}
   
\theoremstyle{remark}
    
\numberwithin{equation}{section}

\usepackage[%
   plainpages=false, 
   colorlinks,
   linkcolor=blue,
   linktocpage, 
   anchorcolor=blue,
   citecolor=blue,
   filecolor=blue,
   urlcolor=blue
   ]{hyperref}%
\newcommand{\tref}[2]{\hyperref[#2]{#1~\ref*{#2}}}
\newcommand{\refp}[1]{%
   \hyperref[#1]{\textup{\lpar}\ref*{#1}\textup{\rpar}}}
\newcommand{\trefp}[2]{%
   \hyperref[#2]{#1~\textup{\lpar}\ref*{#2}\textup{\rpar}}}
\newcommand{\trrefp}[3]{%
   \hyperref[#3]{#1~\ref*{#2}\,{\textup{\lpar}\ref*{#3}\textup{\rpar}}}}
\newcommand{\eref}[2]{\hyperref[#2]{#1~\eqref{#2}}}
\usepackage{enumerate}
\newenvironment{enuma}{%
   \begin{enumerate}[\quad\textup{\lpar}a\textup{\rpar}]}{%
   \end{enumerate}}
\DeclareMathOperator{\conv}{conv}
\DeclareMathOperator{\STAB}{STAB}
\DeclareMathOperator{\FRAC}{FRAC}
\newcommand{\ch}{^\ast} 
\newcommand{\Q}{Q\ch}
\newcommand{\qast}{\Q(A)}
\newcommand{\qastb}{\overline{\Q(A)}}
\newcommand{\qa}{Q(A)}
\newcommand{\qab}{\overline{Q(A)}}
\newcommand{\II}[1][n]{\mc{I}_{#1}} 
\DeclareMathOperator{\vs}{V}
\newcommand{\vsr}{\vs(R)}
\newcommand{\vsp}{\vs(P)}
\newcommand{\vphi}{\varphi}
\newcommand{\seg}{[\xi,\eta]}
\newcommand{\conj}{\{\xi,\eta\}}
\newcommand{\cube}{[0,1]^n}
\newcommand{\parts}{\mathscr{P}} 
\newcommand{\phix}[1][x]{\vphi(I,#1)}

\newcommand{\phiz}{\phix[\zeta]}
\newcommand{\phizk}{\phix[\zeta^k]}

\providecommand{\arxiv}[2][]{\href{http://www.arxiv.org/abs/#2}{arXiv:#2}}

\begin{document}


\title{%
   Addendum to
   ``Vertex adjacencies in the set covering polyhedron''}

\author{N\'estor~E.~Aguilera}
\address{Facultad de Ingenier\'ia Qu\'imica (UNL).
   Santiago del Estero 2829,
   3000 Santa Fe, Argentina.}
\email{nestoreaguilera@gmail.com}

\author{Ricardo~D.~Katz}
\address{CONICET-CIFASIS.
   Bv.~27 de febrero 210 bis,
   2000 Rosario, Argentina.}
\email{katz@cifasis-conicet.gov.ar}

\author{Paola~B.~Tolomei}
\address{CONICET and
   Facultad de Ciencias Exactas, Ingenier\'ia y Agrimensura (UNR),
   Pellegrini 250, 2000 Rosario, Argentina.}
\email{ptolomei@fceia.unr.edu.ar}
  
\begin{abstract}
We study the relationship between the vertices of an up-monotone polyhedron $R$ and those of the polytope  $P$ obtained by truncating $R$ with the unit hypercube. When $R$ has binary vertices, we characterize the vertices of $P$ in terms of the vertices of $R$, show their integrality, and prove that the 1-skeleton of $R$ is an induced subgraph of the 1-skeleton of $P$. We conclude by applying our findings to settle a claim in the original paper.
\end{abstract}

\keywords{Polyhedral combinatorics, set covering polyhedron, vertex adjacency}

\catcode`\@=11
\@namedef{subjclassname@2010}{%
  \textup{2010} Mathematics Subject Classification}
\catcode`\@=12
\subjclass[2010]{90C57, 52B05, 05C65}

\maketitle

\section{Introduction}
\label{sec:intro}

In \cite{AKT} we studied vertex adjacency in the (unbounded version of the) set covering polyhedron associated with a binary matrix $A$:
\begin{equation}\label{equ:qast}
   \qast = \conv(\{x\in\ZZ^n\mid Ax\ge\one, x\ge\zero\}),
\end{equation}
where $\zero$ and $\one$ denote vectors of appropriate dimension with all zeros and all ones components respectively, and $\conv(X)$ denotes the convex hull of the set $X\subset\RR^n$. This polyhedron is the dominant of the set covering polytope associated with $A$:
\begin{equation}\label{equ:qastb}
   \qastb = \conv(\{x\in\ZZ^n\mid Ax\ge\one,\one \ge x \ge \zero\}),
\end{equation}
that is, $\qast = \qastb +\{x\in \RR^n\mid x \ge \zero\}$ where $+$ denotes the Minkowski sum of subsets of $\RR^n$.

Immediately after stating Theorem~2.1 in \cite{AKT}, we made the following claim:

\begin{claim}\label{laclaim}
It can be proved that for any binary matrix $A$ two vertices of $\qast$ are adjacent if and only if they are adjacent in $\qastb$.
\end{claim}

Although this result may seem quite natural, we would like to observe that it is no longer true if we replace $\Q(A)$ by its linear relaxation,
\begin{equation}\label{equ:qa}
   \qa = \{x\in\RR^n \mid Ax\ge\one, x\ge\zero\},
\end{equation}
and $\qastb$ by the corresponding bounded version,
\begin{equation}\label{equ:qab}
   \qab = Q(A)\cap\cube.
\end{equation}
This may be seen by considering the circulant matrix
\begin{equation}\label{equ:matA}
   A  =
      \begin{bmatrix}
      1 & 1 & 0 \\
      0 & 1 & 1 \\
      1 & 0 & 1
      \end{bmatrix}.
\end{equation}
In this case, the vertices of $\qab$ are
\[ (1,1,0),\ (0,1,1),\ (1,0,1),\ (1, 1, 1),\ (1/2,1/2,1/2), \]
and $\xi = (1, 1, 0)$ and $\eta = (0, 1, 1)$ are adjacent in $\qab$ but not in $\qa$. 
Furthermore,
as is readily verified,
in this example $\xi$ and $\eta$ are adjacent in $\Q(A)$,
which means that in general $Q(A)$ does not have the \emph{Trubin property} with respect to $\Q(A)$.%
   \footnote{Let us recall that a polyhedron $R$ has the Trubin property with respect to a polyhedron $P$ contained in $R$ if the $1$-skeleton of $P$ is an induced subgraph of the $1$-skeleton of $R$, see~\cite{Trubin}.}

This is rather surprising since in the special case in which $A$ has precisely two ones per row, i.e., the case in which $A$ is the edge-node incidence matrix of a graph $G$, $\qab$
has the Trubin property with respect to $\qastb$.%
	\footnote{Since the linear relaxation $\FRAC(G)$ of the stable set polytope $\STAB(G)$ has the Trubin property with respect to $\STAB(G)$ (see~\cite{Padberg}), and the function $x \to \one - x$ affinely maps $\overline{Q(A)}$ to $\FRAC(G)$.}

One of the aims of this paper is to prove the validity of \tref{Claim}{laclaim}. Along the road we will establish relationships between the vertices of an up-monotone polyhedron $R$ and those of a polyhedron $Q\subseteq R$ such that the vertices of $R$ belong to $Q$. The results here do not depend on those in~\cite{AKT}, and we think they are interesting by themselves.

This addendum is organized as follows. In \tref{Section}{sec:basicos} we introduce some notation and present basic results concerning vertices and their adjacency in an up-monotone polyhedron. \tref{Section}{sec:bound} is the core of the paper, where we study the effect of cutting with the unit hypercube an up-monotone polyhedron having only binary vertices, first characterizing the vertices of the new polytope (\tref{Corollary}{coro:bin1}) and proving their integrality (\tref{Corollary}{coro:bin2}), and then studying the adjacency of the vertices of the larger polyhedron in the new polytope (\tref{Theorem}{thm:adj}). We conclude by relating our findings to the original article~\cite{AKT} in \tref{Section}{sec:original}.

\section{Some properties of vertices in up-monotone polyhedra}
\label{sec:basicos}

In this section we introduce notation which perhaps is not quite established in the literature, and state a few basic results that are either simple to prove or well-known, and so we will omit most of the proofs.

Let us start with the notation, part of which we have already used.

The set $\{1,\dots,n\}$ is denoted by $\II$, the family of subsets of $\II$ by $\parts$, the $i$-th vector of the canonical base of $\RR^n$ by $\ee_i$, and the scalar product in $\RR^n$ by a dot: $x\cdot y = \sum_{i = 1}^n x_i y_i$.

For $x$ and $y$ in $\RR^n$,
$[x,y] = \conv( \{x, y\})$ represents the (closed) segment joining them,
and we write $x\ge y$ (resp.\ $x > y$) if $x_i\ge y_i$ (resp.\ $x_i > y_i$) for all $i \in \II$ (notice that $x\gneqq y$, i.e., $x\ge y$ and $x\ne y$, does not imply $x > y$).

Given a polyhedron $S$, the set of its vertices is denoted by $\vs(S)$.

Throughout the paper we will assume that
$R\subset \{x\in\RR^n \mid x\ge\zero\}$ is a non-empty polyhedron
which is \emph{up-monotone},%
   \footnote{Or \emph{upper comprehensive} in the nomenclature
   of some authors.}
that is, it satisfies any of the following equivalent conditions:
\begin{itemize}
\item
$x\in R$ and $y\ge x$ imply $y\in R$,

\item
$x\in R$ if and only if $x = y + \mu$ with $y\in\conv(\vsr)$ and $\mu\ge\zero$.
\end{itemize}

Our first result relates vertices and minimality in $R$.

\begin{lem}\label{lem:1}
Assuming $\xi$ and $\eta$ are distinct vertices of $R$ and $x\in R$, we have:

\begin{enuma}
\item\label{lem:1:a}
If $x\le \xi$ then $x = \xi$.

\item\label{lem:1:b}
If $x\in\seg$ and $\mu\ge\zero$ is such $x + \mu\in\seg$, then $\mu = \zero$.
\end{enuma}
\end{lem}

The following proposition is fundamental to our work.

\begin{propo}\label{propo:adj}
If $\vsr = \{\xi = \zeta^1, \eta = \zeta^2,\dots,\zeta^r\}$,
then the following are equivalent:
\begin{enuma}
\item\label{propo:adj:a}
$\xi$ and $\eta$ are adjacent in $R$, that is, there exist $c\in\RR^n$ and $b\in\RR$ such that
$c \cdot x \ge b$
for all $x\in R$,
with equality if and only if $x\in \seg$.

\item\label{propo:adj:b}
There exist $c\in\RR^n$ and $b\in\RR$ such that $c > \zero$ and
\begin{equation}\label{equ:adj:b}
   c \cdot x \ge b  \text{ for all $x\in\conv(\vsr)$, with equality if and only if $x\in \seg$.}
\end{equation}

\item\label{propo:adj:c}
If $x + \mu\in\seg$, where $\mu\ge\zero$ and $x$ is a convex combination of the form $\sum_{k=1}^r \lambda_k \zeta^k$, then
$\lambda_k = 0$ for $k = 3,\dots,r$ (so $x\in\seg$ and $\mu = \zero$).
\end{enuma}
\end{propo}

\begin{proof}
It is easy to show that~\refp{propo:adj:a} implies~\refp{propo:adj:b} and that~\refp{propo:adj:b} implies~\refp{propo:adj:c}. Thus, we next show only that~\refp{propo:adj:c} implies~\refp{propo:adj:a}. We do this by contradiction, so assume \refp{propo:adj:c} holds but $\xi$ and $\eta$ are not adjacent in $R$. Then, the minimal face of $R$ containing $\xi$ and $\eta$ has dimension at least 2. It follows that there exist two points $y',y''\in R$ and $\lambda \in \RR$ such that $0 < \lambda < 1$, $\lambda y' + (1 - \lambda)\,y''\in\seg$ and neither $y'$ nor $y''$ belong to $\seg$. Since $R$ is up-monotone, we can find $x'$ and $x''$ in $\conv(\vsr)$ and $\mu',\mu''\ge\zero$ such that $y' = x' + \mu'$ and $y'' = x'' + \mu''$. Writing $x = \lambda x' + (1 - \lambda)\,x''$ and $\mu = \lambda \mu' + (1 - \lambda)\,\mu''$, we have $x\in\conv(\vsr)$, $\mu\ge\zero$, and $x+\mu = \lambda y' + (1 - \lambda)\,y'' \in \seg$,  so by~\refp{propo:adj:c} we conclude that $\mu = \zero = \mu' = \mu''$, that is, $y'=x'$ and  $y''=x''$. Since $x'$ and $x''$ are in $\conv(\vsr)$, it follows that 
\[
   y'=x'= \sum_{k=1}^r \tau'_k \zeta^k,
   \quad
   y''=x'' = \sum_{k=1}^r \tau''_k \zeta^k,
   \quad
   x = \sum_{k=1}^r
      \big(\lambda \tau'_k + (1 - \lambda)\,\tau''_k\big) \zeta^k,
\]
and again by~\refp{propo:adj:c} 
we must have
\[
   \lambda \tau'_k + (1 - \lambda)\,\tau''_k = 0
   \quad\text{for $k\ne 1,2$,}
\]
i.e., $\tau'_k = \tau''_k = 0$ for $k\ne 1,2$.
Hence, $y'$ and $y''$
are convex combinations of $\xi$ and $\eta$, that is, they are in $\seg$, contradicting the way they have been chosen above.
\end{proof}

\section{Bounding with the unit hypercube}
\label{sec:bound}

We now turn our attention to
studying the relationship between the vertices of the up-monotone polyhedron $R$ and those of $R\cap\cube$.

We omit the proof of the following simple result relating the vertices of two polyhedra in a somewhat more general setting.

\begin{lem}\label{lem:3}
Let $S$ and $T$ be polyhedra such that $S\subset T$.
We have:
\begin{enuma}
\item\label{lem:3:1}
If $\xi\in\vs(T)\cap S$ then $\xi\in\vs(S)$.

\item\label{lem:3:2}
If $\xi$ and $\eta$ are distinct points in $\vs(T)\cap S$ which are adjacent in $T$, then they are also adjacent in $S$.
\end{enuma}
\end{lem}

In the remainder of this section, we will assume that
$R$ is an up-monotone polyhedron satisfying
\begin{subequations}\label{equs:hipo}
\begin{equation}
  \label{equ:vsr}
  \vsr\subset\cube,
\end{equation}
and
$P$ is defined by
\begin{equation}
   \label{equ:p}
   P = R\cap\cube,
\end{equation}
\end{subequations}
so that $\vsr\subset P$.

We will find it convenient to consider the function $\vphi: \parts\times \RR^n \to \RR^n$ defined component-wise by
\begin{equation}\label{equ:phi}
   \phix_i = \begin{cases}
      1 & \text{if $i\in I$,} \\
      x_i & \text{otherwise,}
      \end{cases}
\end{equation}
that is, a projection for each $I\in\parts$.
Notice that $\phix = x$ if $I$ is empty,
and that if $x\in P = R\cap\cube$ then $x\le\phix$ and $\phix\in P$ because $R$ is up-monotone.

The next result says that any vertex of $P$ can be obtained by ``lifting'' a vertex of $R$ via $\vphi$.

\begin{thm}\label{thm:lift}
If $\vphi$ is defined by~\eqref{equ:phi}, then
\begin{equation}\label{eq:lift}
\vsp \subset \{\phiz \mid I\in \parts, \zeta\in\vsr\}.
\end{equation}
\end{thm}

\begin{proof}
Observe that $P = R\cap \{x\in\RR^n \mid x\le\one\}$ because we have assumed $R\subset \{x\in\RR^n \mid x\ge\zero\}$. Then, to prove the theorem it is enough to show that any polyhedron in the sequence: $P_0=R$ and $P_k = P_{k-1} \cap \{x\in\RR^n \mid x_k\le 1\}$ for $k\in \II$, satisfies~\eqref{eq:lift}. We do this by induction. Since $\phix = x$ if $I$ is empty, it is obvious that $P_0$ satisfies~\eqref{eq:lift}. So assume now that $P_{k-1}$ satisfies~\eqref{eq:lift}. Note that the vertices of $P_{k}$ which are not vertices of $P_{k-1}$ coincide with the intersections consisting of a single point of the hyperplane $\{x\in\RR^n \mid  x_{k} = 1\}$ with the relative interior of edges of $P_{k-1}$. Besides, observe that the relative interior of no bounded edge of $P_{k-1}$ can intersect $\{x\in\RR^n \mid  x_{k} = 1\}$ in a single point because that would imply $\xi_k> 1$ for some vertex $\xi$ of $P_{k-1}$, contradicting that $V(P_{k-1})\subset \cube$ (which follows from the fact that $P_{k-1}$ satisfies~\eqref{eq:lift} and $R$ satisfies~\eqref{equ:vsr}). Thus, any vertex of $P_{k}$ which is not a vertex of $P_{k-1}$ is given by the intersection of the relative interior of an unbounded edge of $P_{k-1}$ with $\{x\in\RR^n \mid  x_{k} = 1\}$. Since $R$ is up-monotone, any unbounded edge of $P_{k-1}$ is of the form $\{\xi + \gamma \ee_h \mid \gamma \ge 0\}$, where $\xi$ is a vertex of $P_{k-1}$ and $h\in\{k,\ldots ,n\}$. This completes the proof, because when the intersection of $\{\xi + \gamma \ee_h \mid \gamma \ge 0\}$ with $\{x\in\RR^n \mid  x_{k} = 1\}$ is not empty (i.e., when $h= k$), it consists of a point which can be obtained replacing the $h$-th component of $\xi$ by a one.\qedhere
\end{proof}

The following is an immediate consequence of \tref{Theorem}{thm:lift}.

\begin{coro}\label{coro:ricardo:1}
Suppose
$R$ and $P$ verify~\eqref{equs:hipo} and
$S$ is a polyhedron verifying $\vsr\subset S\subset R$ and
\[ \phiz\in S \text{ for all $I\in\parts$ and $\zeta\in\vsr$.} \]
Then $\vsp\subset S$, that is, $P = R\cap\cube \subset S$.
\end{coro}

So far we have not assumed the integrality of the vertices of $R$, and, for instance, \tref{Theorem}{thm:lift} may be applied to
$R = \qa$ (defined in~\eqref{equ:qa})
and $P = \qab$ (defined in~\eqref{equ:qab}).

Before studying the case $\vsr\subset\BB^n$, where $\BB = \{0,1\}$ denotes the set of binary numbers,  let us state without proof some simple properties relating $\vphi$, binary points and vertices of $R$ and $P = R\cap\cube$.

\begin{lem}\label{lem:vsp}
In the following we assume $I\in \parts$.
\begin{enuma}
\item\label{lem:vsp:a}
If $x\in \BB^n$ then $\phix\in\BB^n$.

\item\label{lem:vsp:b}
If $x\in P\cap\BB^n$ then $x\in\vsp$.

\item\label{lem:vsp:c}
If $x\in R\cap\BB^n$ then $\phix\in\vsp$.

\end{enuma}
\end{lem}

The following result characterizes the vertices of $P$ when the vertices of $R$ are binary.

\begin{coro}\label{coro:bin1}
If $\vsr\subset\BB^n$,
$P = R\cap\cube$,
and $\vphi$ is defined by~\eqref{equ:phi}, then
\[ \vsp = \{\phiz \mid I\in \parts, \zeta\in\vsr\}. \]
\end{coro}

\begin{proof}
One inclusion is given by \trrefp{Lemma}{lem:vsp}{lem:vsp:c},
and the other one by \tref{Theorem}{thm:lift}.\qedhere
\end{proof}

Using \tref{Lemma}{lem:vsp}, it is easy to see now that all vertices of $R\cap\cube$ are binary.

\begin{coro}\label{coro:bin2}
If $\vsr\subset\BB^n$ and $P = R\cap \cube$, then $\vsp\subset\BB^n$.
\end{coro}

We come now to the main result of this work.

\begin{thm}\label{thm:adj}
Assume $\vsr\subset\BB^n$,
$P = R\cap\cube$,
and $\xi$ and $\eta$ are distinct vertices of $R$.

Then, $\xi$ and $\eta$ are adjacent in $P$ if and only if they are adjacent in $R$.
\end{thm}

\begin{proof}
Since $\vsr\subset R \cap \BB^n\subset  R\cap \cube = P\subset R$, one implication is given by \trrefp{Lemma}{lem:3}{lem:3:2}.

For the other, if $\xi$ and $\eta$ are adjacent in $P$ there exist $c\in\RR^n$ and $b\in\RR$ such that
\begin{equation}\label{equ:adj}
   c\cdot x \ge b
   \text{ for all $x\in P$,
      with equality if and only if $x\in\seg$.}
\end{equation}

If $c > \zero$ the result follows by \trrefp{Proposition}{propo:adj}{propo:adj:b} because $\conv(\vsr) \subset P$.
So let us assume that the set
\begin{equation}\label{equ:I}
   I = \{i\in \II \mid c_i\le 0\}
\end{equation}
is not empty. In this case, to prove that $\xi$ and $\eta$ are adjacent in $R$, we show that \trrefp{Proposition}{propo:adj}{propo:adj:c} is satisfied. Then, setting 
\[
\vsr = \{\xi = \zeta^1, \eta = \zeta^2,\dots,\zeta^r\} , 
\]
let a convex combination of the vertices of $R$ of the form
\begin{equation}\label{equ:x}
z = \sum_{k=1}^r \lambda_k \zeta^k
\end{equation}
and $\mu\ge\zero$ be such that
\begin{equation}\label{equ:xmu}
   z + \mu\in\seg.
\end{equation}

Notice that for any $x\in\seg \subset P$, since $\phix\in P$ for such $x$ and $c_i \le 0$ for $i\in I$, by~\eqref{equ:adj} we have
\[
  b = c\cdot x
   = \sum_{i\notin I} c_i x_i + \sum_{i\in I} c_i x_i
   \ge \sum_{i\notin I} c_i x_i + \sum_{i\in I} c_i
   = c\cdot \phix
   \ge b,
\]
and so $c\cdot \phix =b$. It follows that $\phix\in\seg$ by~\eqref{equ:adj}, and then that $x = \phix$ by~\trrefp{Lemma}{lem:1}{lem:1:b}. In particular, we conclude that
\begin{equation}\label{equ:seg}
   x_i = 1 \text{ for all $x\in\seg$ and all $i\in I$.}
\end{equation}

Letting
\begin{equation}\label{equ:y}
   y = \sum_{k=1}^r \lambda_k \phizk,
\end{equation}
and defining $\tau$ component-wise by
\begin{equation}\label{equ:tau}
   \tau_i = \begin{cases}
      0 & \text{if $i\in I$,} \\
      \mu_i & \text{otherwise},
      \end{cases}
\end{equation}
by checking the components and using~\eqref{equ:seg} and~\eqref{equ:xmu},
we see that $y + \tau = z + \mu$.
Moreover, as $\tau_i = 0$ for $i\in I$ and $c_i > 0$ for $i\notin I$, by~\eqref{equ:adj} we obtain
\[
b = c\cdot (z + \mu) = c\cdot (y + \tau) \ge c\cdot y \ge b,
\]
and therefore $\tau_i = 0$ for $i\notin I$, that is, $\tau = \zero$. Thus, we have $y = z + \mu \in\seg$.

By \trrefp{Lemma}{lem:vsp}{lem:vsp:c}, $\phizk\in\vsp$ for all $k$, and since $y\in\seg$, \eqref{equ:y} and the adjacency of $\xi$ and $\eta$ in $P$ imply now that for each $k = 1,\dots,r$, either $\phizk\in\conj$ or $\lambda_k = 0$. Using \trrefp{Lemma}{lem:1}{lem:1:a} and the fact that $\zeta^k \le \phizk$, we see that $\phizk\in\conj$ implies $\zeta^k\in\conj$. Thus, in~\eqref{equ:x} we must have either $\zeta^k\in\conj$ or $\lambda_k = 0$, that is,  \trrefp{Proposition}{propo:adj}{propo:adj:c} is satisfied.\qedhere
\end{proof}

\begin{coro}\label{coro:ricardo:2}
Suppose
$\vsr\subset\BB^n$,
$S$ is a polyhedron such that $\vsr\subset S\subset R$ and
$\phiz\in S$ for all $I\in\parts$ and $\zeta\in\vsr$,
and $\xi$ and $\eta$ are distinct vertices of $R$ (and hence of $S$ by \trrefp{Lemma}{lem:3}{lem:3:1}).

Then $\xi$ and $\eta$ are adjacent in $S$ if and only if they are adjacent in $R$.
\end{coro}

\begin{proof}
Let us start by assuming that $\xi$ and $\eta$ are adjacent in $S$.
By \tref{Corollary}{coro:ricardo:1},
$P = R\cap\cube\subset S$,
and adjacency in $S$ implies adjacency in $P$ by \trrefp{Lemma}{lem:3}{lem:3:2} because
$\conj\subset \vs(S)\cap P$ and $P\subset S$.
So, by \tref{Theorem}{thm:adj}, $\xi$ and $\eta$ are adjacent in $R$.

On the other hand, if $\xi$ and $\eta$ are adjacent in $R$,
their adjacency in $S$ follows again from \trrefp{Lemma}{lem:3}{lem:3:2}, as $\vsr\subset S\subset R$.\qedhere
\end{proof}

The conclusion of relevance in
\tref{Theorem}{thm:adj}
is that vertices in the up-monotone polyhedron $R$ which are adjacent in $P = R\cap\cube$ are also adjacent in $R$, provided that $\vsr\subset\BB^n$. As we have seen in the \hyperref[sec:intro]{Introduction}, we cannot discard this hypothesis: if $A$ is given by~\eqref{equ:matA}, then $R = \qa$ has just one fractional vertex, and the vertices $(1,1,0)$ and $(0,1,1)$ are adjacent in $P = \qab$ but not in $R$.

\section{The claim in the original article}
\label{sec:original}

The polyhedron
\[ R = \qast = \conv(\{x\in\ZZ^n\mid Ax\ge\one, x\ge\zero\}), \]
is up-monotone and $\vsr\subset \BB^n$, and it is simple to see that
\[ S = \qastb = \conv(\{x\in\ZZ^n\mid Ax\ge\one,\one \ge x \ge \zero\}) \]
satisfies the hypothesis in \tref{Corollary}{coro:ricardo:2}, proving \tref{Claim}{laclaim}.

Notice that
\tref{Corollary}{coro:ricardo:1}
implies $\qast \cap \cube\subset \qastb$.
On the other hand,
clearly $\{x\in\ZZ^n\mid Ax\ge\one,\one \ge x \ge \zero\}$
is contained in both $\qast$ and $\cube$, so
\[ \qastb = \qast \cap \cube, \]
and actually \tref{Theorem}{thm:adj} may be applied directly.

\section*{Acknowledgements}
\label{sec:acknowledge}

The authors are very grateful to the anonymous reviewer for his comments and suggestions on a previous version of this paper. This work was partially supported by grant PIP 112-201101-01026 from Consejo Nacional de Investigaciones Cient\'ificas y T\'ecnicas (CONICET), Argentina.



\end{document}